\documentclass[11pt]{amsart}
\usepackage[margin=3cm]{geometry}
\usepackage{xcolor}
\newtheorem{Thm}{Theorem}
\newtheorem*{Thm*}{Theorem}

\newtheorem{cor}{Corollary}

\global\long\def\epsilon{\varepsilon}

\begin{document}
\date{06 July, 2017}

\title{The signs of the Stieltjes constants associated with the Dedekind zeta function }
\author{Sumaia Saad Eddin}

\maketitle
{\def\thefootnote{}
}

\begin{abstract}
The Stieltjes constants $\gamma_n(K)$ of a number field $K$ are the coefficients of the Laurent expansion of the Dedekind zeta function $\zeta_K(s)$ at its pole $s=1$. In this paper, we establish a similar expression of  $\gamma_n(K)$ as Stieltjes obtained in 1885 for $\gamma_n(\mathbb Q)$. We also study the signs of $\gamma_n(K)$.  
\end{abstract}

\section{Introduction}
For a number field $K$, the 
\emph{Dedekind zeta function} is defined 
$$
\zeta_K(s)=\sum_{\mathfrak{a}} \frac{1}{N{\mathfrak{a}}^{s}}
=\prod_{\mathfrak{p}}\frac{1}{1-N{\mathfrak{p}}^{-s}}, \quad \Re(s)>1.
$$
Here, $\mathfrak{a}$ runs over non-zero ideals in $\mathcal{O}_K$, the ring of integers of $K$,
$\mathfrak{p}$ runs over the prime ideals in $\mathcal{O}_K$ and $N{\mathfrak{a}}$ is the norm
of $\mathfrak{a}$.
It is known that $\zeta_K(s)$ can be analytically continued to $\mathbb{C} - \{1\}$,
and that at $s=1$ it has a simple pole, with residue $\gamma_{-1}(K)$ given by 
$$ \gamma_{-1}(K)= \frac{2^{r_1}(2\pi)^{r_2}h(K)R(K)}{\omega(K)\sqrt{|d(K)|}},$$
 where $r_1$ denotes the number of real embeddings of $K$, $r_2$ is the number of complex embeddings of $K$, $h(K)$  is the class number of $K$, $R(K)$ is the regulator of $K$, $\omega(K)$ is the number of roots of unity contained in $K$ and $d(K)$ is the discriminant of the extension $K/\mathbb{Q}$. Further, the Laurent expansion of $\zeta_K(s)$ at $s=1$ is 
\begin{equation}
\label{eq1}
\zeta_K(s)=
\frac{\gamma_{-1}(K)}{s-1}+\sum_{n\geq 0}\gamma_n(K)(s-1)^n.
\end{equation}
The coefficients  $\gamma_n(K)$ are sometimes called the Stieltjes constants associated with the Dedekind zeta function.  In \cite{HA}, they are called by higher Euler's constants of $K$. While
the constant $\gamma_K=\gamma_0(K)/\gamma_{-1}(K)$ is called the {\it Euler-Kronecker constant} in
Ihara \cite{Ihara} and Tsfasman \cite{Tsfasman}.\\

In case $K=\mathbb{Q}$, the Laurent expansion of the Riemann zeta function $\zeta(s)$ at its pole $s=1$ is given by \begin{equation*}
\zeta(s)=\frac{1}{s-1}+\sum_{n\geq 0}\gamma_n(s-1)^n,
\end{equation*}
where
\begin{equation}
\label{eq01}
\gamma_n= \frac{(-1)^n}{n!}\lim\limits_{x\rightarrow \infty}\left(\sum\limits_{m=1}^{x}\frac{(\log m)^n }{m}-\frac{(\log x)^{n+1}}{(n+1)}\right).
\end{equation} 
Stieltjes in 1885 was the first to propose this definition of $\gamma_n$, for this reason these constants are called today by his name. 
 The asymptotic behaviour of $\gamma_n$, as $n\rightarrow \infty$ has been widely studied by many authors ( for instance: Briggs~\cite{B}, Mitrovi\`c~\cite{MI} , Israilov~\cite{I1}, Matsuoka~\cite{M} and more recently Coffey~\cite{C} and \cite{C1}, Knessl and Coffey~\cite{K-C}, Adell~\cite{A}, Adell and Lekuona~\cite{A1} and Saad Eddin~\cite{S}). Their main interest is focused on the growth, the sign changes of the sequence $(\gamma_n)$ and on giving explicit upper estimates for $|\gamma_n|$. Moreover, they obtained relations between this sequence and the zeros of $\zeta (s)$ ( see \cite{M}, \cite{S1}). \\

In this paper we are interested in the Stieltjes coefficients $ \gamma_n(K)$ for the Dedekind zeta function. We first give the following formula of $ \gamma_n(K)$ which is similar to Stieltjes's formula given by Eq~\eqref{eq01}. 
\begin{Thm}
\label{thm1}
For any $n\geq 1$, we have 
\begin{equation*}
\gamma_n(K)=
\frac{(-1)^n}{n!}\lim_{x \rightarrow \infty} \left( \sum_{N{\mathfrak{a}}\leq x} \frac{(\log N{\mathfrak{a})^n}}{N{\mathfrak{a}}}-\gamma_{-1}(K)\frac{(\log x)^{n+1}}{n+1}\right),
\end{equation*}
and 
\begin{equation*}
\gamma_0(K)=\lim_{x \rightarrow \infty} \left( \sum_{N{\mathfrak{a}}\leq x} \frac{1}{N{\mathfrak{a}}}-\gamma_{-1}(K)\log x\right)+\gamma_{-1}(K).
\end{equation*}
\end{Thm}
This result seems similar to another result  obtained by Hashimoto et al~\cite{HA} for the higher Euler-Selberg constants. Despite a considerable effort the author have not been able to find Theorem~\ref{thm1} in the literature. \\

In 1962, Mitrovi\`c~\cite{MI}  studied the sign changes of the constants $\gamma_n$ and prove that; Each of the inequalities 
$$ \gamma_{2n}>0, \ \gamma_{2n}<0, \ \gamma_{2n-1}>0, \ \gamma_{2n-1}<0,$$ 
holds for infinitely many $n$. In \cite{M}, Matsuoka gave precise conditions for the sign of $\gamma_n$. By the same techniques used in~\cite{MI}, we prove that  
\begin{Thm}
\label{Thm2}
For the coefficients in the expansion \eqref{eq1}, each of the inequalities 
$$ \gamma_{2n}(K)>0, \ \gamma_{2n}(K)<0, \ \gamma_{2n-1}(K)>0, \ \gamma_{2n-1}(K)<0,$$ 
holds for infinitely many $n$.  
\end{Thm}
It immediately follows that 
\begin{cor}
Infinitely many $\gamma_n(K)$ are positive and infinitely many are negative. 
\end{cor}
\section{Proofs}
\begin{proof}[Proof of Theorem \ref{thm1}]
By Eq~\eqref{eq1}, we note that 
\begin{eqnarray}
\label{eq2}
\zeta_K(s)-\frac{\gamma_{-1}(K)s}{s-1}&=&\zeta_K(s) -\frac{\gamma_{-1}(K)}{s-1} -\gamma_{-1}(K) \nonumber
\\&=&
\sum_{n\geq 0} \alpha_{n}(K)(s-1)^n, 
\end{eqnarray}
where $\alpha_0(K)=\gamma_{0}(K)-\gamma_{-1}(K)$ and $\alpha_{n}(K)=\gamma_{n}(K)$ for $n\geq 1$. By the definition of $\zeta_K(s)$, we write 
\begin{equation*}
\zeta_K(s)=\int_{1^-}^{+\infty} \frac{dN_K(t)}{t^s}=s\int_{1^-}^{+\infty} \frac{N_K(t)}{t^{s+1}}\, dt, 
\end{equation*}
where $$ N_K(t)=\sum_{N{\mathfrak{a}}\leq t}1.$$
Then, we get 
\begin{equation}
\label{eq3}
\zeta_K(s)-\frac{\gamma_{-1}(K)s}{s-1}=s \int_{1^-}^{+\infty} \frac{N_K(t)-\gamma_{-1}(K)t}{t^{s+1}}\, dt.
\end{equation}
Put $ \sum_{n\geq 0} \alpha_{n}(K)(s-1)^n=h(s)$. From Eq\eqref{eq2} and Eq~\eqref{eq3}, we have  
\begin{equation*}
h(s)=s \int_{1^-}^{+\infty} \frac{N_K(t)-\gamma_{-1}(K)t}{t^{s+1}}\, dt.
\end{equation*}
From \cite[Satz 210]{LAN}, we have $ N_K(t)=\gamma_{-1}(K)t+\mathcal{O}\left( t^{1-1/m}\right)$, where $m$ is the degree of $K$ and $\mathbb{Q}$. For $\Re{s}>1-1/m$, it is easily seen that the $n$-th derivative of $h(s)$ at $s=1$ is 
\begin{equation}
\label{eq4}
h^{(n)}(1)=n! \alpha_{n}(K)= (-1)^n(I_1-I_2),
\end{equation}
where
\begin{equation*}
I_1=\int_{1^-}^{+\infty}  N_K(t)\left(\frac{\log^n t-n(\log t)^{n-1}}{t^{2}}\right) \, dt,
\end{equation*}
and
\begin{equation*}
I_2=\gamma_{-1}(K)\int_{1^-}^{+\infty} \frac{\log^n t-n(\log t)^{n-1}}{t} \, dt.
\end{equation*}
On the other hand, we have 
\begin{equation*}
\sum_{{N{\mathfrak{a}}\leq x}}\frac{(\log{N{\mathfrak{a}}})^n }{{N{\mathfrak{a}}}}= \int_{1^-}^{x}\frac{\log^n t}{t}\, dN_K(t)=
 N_K(x)\frac{\log^n x}{x}+\int_{1^-}^{x}N_K(t)\left(\frac{\log^n t-n(\log t)^{n-1}}{t^{2}}\right) \, dt.
\end{equation*}
Thus, we get 
\begin{equation*}
\int_{1^-}^{x}N_K(t)\left(\frac{\log^n t-n(\log t)^{n-1}}{t^{2}}\right) \, dt=\\
\sum_{{N{\mathfrak{a}}\leq x}}\frac{(\log{N{\mathfrak{a}}})^n }{{N{\mathfrak{a}}}}
-N_K(x)\frac{\log^n x}{x}.
\end{equation*}
Again using the fact that $ N_K(t)=\gamma_{-1}(K)t+\mathcal{O}\left( t^{1-1/m}\right)$, we find 
\begin{equation*}
\int_{1^-}^{x}N_K(t)\left(\frac{\log^n t-n(\log t)^{n-1}}{t^{2}}\right) \, dt= 
\sum_{{N{\mathfrak{a}}\leq x}}\frac{(\log{N{\mathfrak{a}}})^n }{{N{\mathfrak{a}}}}
-\gamma_{-1}(K)\log^n x+\mathcal{O}\left( \frac{\log^n x}{x^{1/m}}\right).
\end{equation*}
Taking $x \rightarrow +\infty$, the above becomes 
\begin{equation}
\label{eq5}
I_1=
\lim_{x \rightarrow +\infty}\left[\sum_{{N{\mathfrak{a}}}\leq x}\frac{(\log{N{\mathfrak{a}}})^n }{{N{\mathfrak{a}}}}
-\gamma_{-1}(K)\log^n x\right].
\end{equation}
Now, notice that 
\begin{equation}
\label{eq6}
I_2=
\lim_{x \rightarrow +\infty}\left[\gamma_{-1}(K)\frac{(\log x)^{n+1}}{n+1}-\gamma_{-1}(K)\log^n x \right]
\end{equation}
From Eq \eqref{eq4} and \eqref{eq5} and\eqref{eq6}, we conclude that, for $n\geq 1$,
\begin{equation*}
\gamma_{n}(K)=\alpha_{n}(K)=
\frac{(-1)^n}{n!}\lim_{x \rightarrow \infty} \left( \sum_{N{\mathfrak{a}}\leq x} \frac{(\log N{\mathfrak{a})^n}}{N{\mathfrak{a}}}-\gamma_{-1}(K)\frac{(\log x)^{n+1}}{n+1}\right)
\end{equation*}
and $\gamma_{0}(K)=\alpha_0(K)+\gamma_{-1}(K)$. This completes the proof.
\end{proof}
\begin{proof}[Proof of Theorem \ref{Thm2}]
To prove Theorem~\ref{Thm2}, we apply the same technique used in \cite{MI}.  
Let $C$ be the set of all positive integers $n$ such that $ \gamma_{n}(K)\neq 0$. Define
\begin{equation*}
C_1=\left \{n : \quad \gamma_{n}(K)\neq 0 \text{ and} \ (-1)^n=1 \right\}
\end{equation*}
\begin{equation*}
C^{-}_1=\left \{n : \quad \gamma_{n}(K)< 0 \text{ and} \ (-1)^n=1 \right\},
\end{equation*}
\begin{equation*}
C^{+}_1=\left \{n : \quad \gamma_{n}(K)> 0 \text{ and} \ (-1)^n=1 \right\},
\end{equation*} 
and 
\begin{equation*}
C_2=\left \{n : \quad \gamma_{n}(K)\neq 0 \text{ and} \ (-1)^n=-1 \right\},
\end{equation*}
\begin{equation*}
C^{-}_2=\left \{n : \quad \gamma_{n}(K)< 0 \text{ and} \ (-1)^n=-1 \right\},
\end{equation*}
\begin{equation*}
C^{+}_2=\left \{n : \quad \gamma_{n}(K)> 0 \text{ and} \ (-1)^n=-1 \right\}.
\end{equation*}
From \cite{REI}, we have 
$$\zeta_K(s)-\frac{\gamma_{-1}(K)}{s-1}$$ 
is an entire \textit{transcendental} function. So the cardinal number of the set $C$ is equal to the cardinal number of the set of all positive integers $\aleph_0$. Then, we can write 
\begin{equation*}
\zeta_K(s)-\frac{\gamma_{-1}(K)}{s-1}=
\Bigl(\sum_{n\in C^{-}_1}+\sum_{n\in C^{+}_1}+\sum_{n\in C^{-}_2} +\sum_{n\in C^{+}_2}\Bigr)\gamma_{n}(K)(s-1)^n.
\end{equation*}
Replacing $s$ by $t+1$ and then by $-t+1$ in the above. Adding and then subtracting  the results, we find that 
\begin{equation}
\label{eq7}
\zeta_K(t+1)+\zeta_K(-t+1)=2\Bigl(\sum_{n\in C^{-}_1} +\sum_{n\in C^{+}_1}\Bigr) \gamma_{n}(K) t^n, 
\end{equation}
and 
\begin{equation}
\label{eq8}
\zeta_K(t+1)-\zeta_K(-t+1)-\frac{2\gamma_{-1}(K)}{t}=\\2\Bigl(\sum_{n\in C^{-}_2} +\sum_{n\in C^{+}_2} \Bigr)\gamma_{n}(K) t^n.
\end{equation}
Taking $ t=2m+1$ with $m>0$ and using the fact that the $\zeta_K(s)$ vanishes at all negative even integers. We find the left hand side of Eq~\eqref{eq7} approaches to $1$ when $m \rightarrow +\infty$. It follows that the right hand side of this equation can't be polynomial. Therefore the cardinal of the set $C_1$ is $ \aleph_0$. On the other hand, if we assume that the cardinal of the set $C^{-}_1$ is less than $\aleph_0$. Then the right hand side of Eq~\eqref{eq7} approaches $+\infty$. Similarly, if the cardinal of the set $C^{+}_1$ is less than $\aleph_0$. Then the right hand side of Eq~\eqref{eq7} approaches $-\infty$, this leads to a contradiction. We thus conclude that the cardinal of the sets $C^{-}_1$ and $C^{+}_1$ are $\aleph_0$. By a similar argument, we show that the cardinal of the sets $C^{-}_2$ and $C^{+}_2$ are $\aleph_0$. That completes the proof. 
\end{proof}

\section*{Acknowledgement} 
The author would like to thank Professor Kohji Matsumoto for his valuable comments on an earlier version
of this paper. The author is supported by the Japan
Society for the Promotion of Science (JSPS) `` Overseas researcher under Postdoctoral Fellowship of JSPS''. Part of this work was done while the author was supported by the Austrian Science
Fund (FWF) : Project F5507-N26, which is part
of the special Research Program  `` Quasi Monte
Carlo Methods : Theory and Application''.

\medskip\noindent {\footnotesize Graduate School of Mathematics, Nagoya University,
Furo-cho, Chikusa-ku, Nagoya, Aichi 464-8602, Japan.\\
e-mail: {\tt saad.eddin@math.nagoya-u.ac.jp}}

\end{document}